\documentclass[oneside,english]{amsart}
\usepackage[T1]{fontenc}
\usepackage[latin9]{inputenc}
\usepackage{amsthm}
\usepackage{amstext}
\usepackage{amssymb}
\usepackage{esint}
\usepackage[hmargin=3cm,vmargin=4cm,bindingoffset=0.5cm]{geometry}

\makeatletter
\theoremstyle{plain}
\newtheorem{thm}{\protect\theoremname}
  \theoremstyle{plain}
  \newtheorem{prop}[thm]{\protect\propositionname}
  \theoremstyle{remark}
  \newtheorem{rem}[thm]{\protect\remarkname}
  \theoremstyle{plain}
  \newtheorem{lem}[thm]{\protect\lemmaname}

\usepackage{chngcntr}

\makeatother

\usepackage{babel}
  \providecommand{\lemmaname}{Lemma}
  \providecommand{\propositionname}{Proposition}
  \providecommand{\remarkname}{Remark}
\providecommand{\theoremname}{Theorem}

\usepackage{fancyhdr}

\begin{document}

\title{The number of cycles in random permutations without long cycles is
asymptotically Gaussian}

\author{Volker Betz}
\address[Volker Betz]{Technische Universit\"at Darmstadt}
\email{betz@mathematik.tu-darmstadt.de}
\author{Helge Sch\"afer}
\address[Helge Sch\"afer]{Technische Universit\"at Darmstadt}
\email{hschaefer@mathematik.tu-darmstadt.de}

\thanks{H.S. is supported by a PhD scholarship from Deutsche Telekom Stiftung.}
\thanks{The authors gratefully acknowledge support by the DFG project BE 5267/1-1.}
\thanks{The authors wish to thank Dirk Zeindler for helpful discussions.}

\begin{abstract}
For uniform random permutations conditioned to have no long cycles, 
we prove that the total number of cycles satisfies a central limit theorem. 
Under additional assumptions on the asymptotic behavior of the set of 
allowed cycle lengths, we derive asymptotic expansions for the corresponding 
expected value and variance. 
\end{abstract}


\maketitle

\section{Introduction}

It is a classical result about random permutations proved by Goncharov
(\cite{goncharov1944some,goncharov1944field})
that the number of cycles of a uniform random permutation is 
asymptotically Gaussian after suitable normalization. More precisely, for a permutation $\sigma \in \mathcal S_n$, write 
$c_j(\sigma)$ for the number of cycles of length $j$ in $\sigma$ 
and $C(\sigma) = \sum_{j=1}^\infty c_j(\sigma)$ for the total number of cycles. Let $\mathbb P_n$ denote the uniform distribution on 
$\mathcal S_n$, then for all $z \in \mathbb R$, 
\begin{equation}
\lim_{n \to \infty} \mathbb P_n (C - \log n \leq z \sqrt{\log n}) = \Phi(z),
\label{clt classical}
\end{equation}
where $\Phi$ is the distribution function of the standard normal distribution. 

A particularly instructive proof of this result was given
in \cite{shepp1966ordered}, where it is observed that the 
joint distribution of the $c_j$, $j \leq n$, under $\mathbb P_n$ is 
equal to the distribution of independent random variables 
$\tilde c_j$, $j < \infty$, conditional on $\sum j \tilde c_j = n$,
where $\tilde c_j$ is Poisson distributed with parameter $z^j/j$ for 
arbitrary $z \in (0,1)$. In the framework of statistical mechanics,
$z$ can thus be viewed as the chemical potential and the 
joint distribution of the $\tilde c_j$ as the grand 
canonical ensemble corresponding to the measures $\mathbb P_n$. 
Noting that the limit $n \to \infty$ corresponds to $z \to 1$, 
Equation \eqref{clt classical} is then easy to guess 
and not very hard to prove. Indeed, much more is true:
Delaurentis and Pittel \cite{DP85} 
show that for $C_{n,t}(\sigma) = \sum_{j=1}^{n^t} 
c_j(\sigma)$, the joint distribution of 
$((\log n)^{-1/2} (C_{n,t} - t \log n))_{0 \leq t \leq 1}$ converges to 
standard Brownian motion on the time interval $[0,1]$. Goncharov \cite{goncharov1944some} 
and Kolchin \cite{Kol71} 
show that the finite dimensional distributions of the sequence 
$(c_j)$ converge to independent Poisson variables $C_j$ with parameters $1/j$. Arratia and Tavar\'e \cite{AT92} significantly 
improve this result: they show that the variation distance between the 
joint distribution of $(c_j)_{j \leq b(n)}$ and the independent Poisson variables 
$(C_j)_{j \leq b(n)}$ converges to zero if and only if $\lim_{n \to \infty} b(n)/n = 0$;
they even show the same result when considering any $b(n)$ indices instead of the 
$b(n)$ first ones. So,  the uniform measure on 
permutations is very well understood. 

A natural next step is to study conditional measures, i.e.\ 
the uniform distribution on certain subsets of $\mathcal S_n$. 
We will be interested 
in permutations where only certain cycle lengths are allowed. Here, 
Theorem 3 in  \cite{AT92} gives some interesting, partly more general answers: if we 
consider the uniform distribution on permutations such that for all $j$ from a subset 
$J \subset \{1, \ldots, b(n)\}$, the cycle counts $c_j$ 
have prescribed values $\bar c_j$, the variation distance between the 
remaining cycle counts and the corresponding independent Poisson variables $C_j$ still converges to zero
as $n \to \infty$. In particular, we can condition $\bar c_j=0$ for as many $j \leq b(n)$
as we like and still get the Poisson convergence. However, the results of 
\cite{AT92} do not allow us to condition on $\bar c_j = 0$ for $j > b(n)$, i.e.\ we cannot
force large cycles to disappear. Conditioning on the absence of {\em some} large 
cycles is the content of the theory of $A$-permutations. Here, a set $A \subset \mathbb N$ 
is fixed, and the sequence of uniform measures on the sets of 
permutations $T_n(A)$ on $\{1, \ldots n\}$ and with cycle lengths in $A$ is studied. 
Under the condition that for each $n$ the number of permutations in 
$T_n(A)$ is large enough (precisely, of order $(n-1)! n^{\alpha}$ for some $\alpha \in (0,1]$), 
Yakimiv shows the Poisson convergence of finite dimensional distributions \cite{Yak07} 
and the central limit theorem for the total number of cycles \cite{Yak08}. The latter paper 
also contains a long list of references to the vast literature on $A$-permutations. 

The situation is different when the 
set $A$ depends on $n$ and is such that long cycles are forbidden. 
For simplicity think of conditioning on the absence of cycles longer than 
$n^\alpha$ where $0 < \alpha < 1$. 
To see why this drastically changes the situation, recall that for 
large $n$ under the uniform measure on permutations, the fraction of indices in cycles 
of order $o(n)$ converges to zero. So, even though macroscopic (order $n$) cycles do not 
contribute significantly to the cycle count, they absorb the majority of indices. If we 
exclude them, the total number of cycles must grow much faster than logarithmically. 
Indeed, we will show below that when excluding cycles longer than $n^\alpha$, the average
number of cycles grows like $n^{1-\alpha}$. 
Since all of the above mentioned methods in essence rely on the fact 
that the total cycle number is well approximated by a sum of independent 
Poisson random variables $\tilde c_j$ with parameters $1/j$, they all  
predict a logarithmic growth of the cycle count. So none of them can possibly apply 
to the situation of restricted cycle lengths. It is thus interesting that the central limit 
theorem for the total number of cycles is robust enough to carry over to the case
of permutations without long cycles. The proof of this statement (for rather 
general restrictions on cycle lengths) is the main result of the 
present paper. 

A somewhat complementary result is obtained by Nikeghbali, Storm, and Zeindler
\cite{NiStZei}. They consider permutations of $n$ elements 
having only cycle lengths in a set $A_n \subset \mathbb N$ 
under the assumption that $\max  (\left\{1,2,...,n\right\} \setminus A_n) = o(n)$, i.e.\ $A_n$ 
contains all the 'long' cycles. They obtain a functional central limit theorem, again 
with logarithmic average cycle number. Indeed, they can treat slightly more general 
types of permutations by including cycle weights: the probability of a permutation $\sigma$ is modified  by a Boltzmann factor 
$\exp(-\sum \alpha_j c_j(\sigma))$, where the $\alpha_j$ correspond to 
the energies of cycles of length $j$. The case of constant 
$\alpha_j = \alpha$ corresponds to the Ewens distribution with 
parameter $\exp(\alpha)$, see \cite{Han90} for the first proof of a  functional central limit theorem in that case. 

A different use of cycle weights actually 
leads to the two only other situations that we are aware of 
where the typical number of cycles is {\em not} logarithmic and satisfies a 
central limit theorem. Firstly, Maples, Nikeghbali, and Zeindler \cite{maples2012number} consider
the generalized Ewens measure with cycle weights of the form $\prod j^{\alpha c_{j}(\sigma)}$ with $\alpha > 0$.
Secondly, Bogachev and Zeindler \cite{BZ15} 
replace the 
sequence $\alpha_j$ by a doubly indexed sequence 
$\alpha_{j,n}$. Under suitable assumptions on the asymptotics of the 
$\alpha_{j,n}$, they show 
(among many other things) that the number of cycles   
is of order $n$ and satisfies a central limit theorem. In particular, 
this implies that,  contrary to the situation in the present paper, 
there is a positive fraction of points in finite cycles.  
The latter property motivates the name 'surrogate spatial random 
permutations' in the title of \cite{BZ15}: while for random 
permutations it is in general rare to see a fraction of indices in 
finite cycles due to very strong entropic effects, this phenomenon 
becomes the norm when we add a spatial component to the model. 
These spatial random permutations are still relatively little 
understood, but appear to have many intriguing properties such as 
a phase transition from a regime without long cycles
to one with a mixture of long and short cycles 
(proved for a special case in \cite{BU11}) and a rich geometry of 
the set of points in long cycles (see \cite{GLU,Be14} for some 
simulations). We refer the reader to \cite{Be14} and the references 
therein for more information about these models.

\section{Results}

Let $\alpha: \mathbb N \to \mathbb N$, and consider 
\[
S_{n,\alpha}:=\left\{ \sigma\in S_{n}:\text{the cycles of }\sigma\text{ are not longer than }\alpha\left(n\right)\right\}, 
\]
the set of permutations with $n$ elements and all cycles shorter than $\alpha(n)$.
Let $\mathbb{P}_{n,\alpha}$ denote the uniform distribution on $S_{n,\alpha}$, and 
write 
\[
C(\sigma) = \sum_{j=1}^n c_j(\sigma)
\] 
for the sum of cycle counts $c_j(\sigma)$.
For $\delta > 0$ and all $w \in (1-\delta, 1+\delta)$, let $x_{n,\alpha}(w)$ be the unique positive solution of the equation 
\begin{equation}
\sum_{j=1}^{\alpha\left(n\right)}x_{n,\alpha}^{j}\left(w\right)=\frac{n}{w}.\label{eq:saddle}
\end{equation}
Equivalently, $x_{n,\alpha}$ is the unique root of the polynomial 
\[
p_{n,\alpha,w}(x) = w x^{\alpha(n)+1} - (w + n) x + n
\]
that is contained in the interval $(1,\infty)$. 
Let us further write 
\begin{equation}
m_{n,\alpha} := \sum_{j=1}^{\alpha(n)} \frac{x_{n,\alpha}^j(1)}{j}, 
\quad v_{n,\alpha} := m_{n,\alpha} + n \frac{x'_{n,\alpha}(1)}{x_{n,\alpha}(1)},
\label{mean and variance}
\end{equation}
where $x'_{n,\alpha}$ denotes the derivative of $x_{n,\alpha}$. 
Our main result is 

\begin{thm} \label{main thm}
	Assume that
	\begin{equation*}
	\liminf_{n \to \infty} \alpha(n) \geq 4
	\end{equation*}
	and  
	\begin{equation} \label{growth condition}
	\limsup_{n \to \infty} \frac{\alpha(n)}{n} \log(n) (\log(\log(n)))^2 
	< \frac{1}{12 \pi^2 \mathrm e}.
	\end{equation}
	Then for all $z \in \mathbb R$, 
	\[
	\lim_{n \to \infty} \mathbb P_{n,\alpha} \Big( 
	C - m_\alpha(n) \geq z \sqrt{v_{n,\alpha}} \Big) = \Phi(z).
	\]
	Here, $\Phi$ is the distribution function of the standard normal distribution. 
\end{thm}

Note that there is an extra term in the variance in equation \eqref{mean and variance}. 
So in contrast to the classical situation, asymptotic mean and variance need not be of 
the same order. Under some additional assumptions on the function $\alpha$, 
we can derive asymptotic expressions for $m_{n,\alpha}$ and $v_{n,\alpha}$ and see that 
they are indeed of different order. For this, recall first the 
concept of asymptotic expansion.
Let $(\phi_{k})_{k \in \mathbb N}$ be a sequence of functions
such that for all $k$
\[
\phi_{k+1}\left(z\right)=o\left(\phi_{k}\left(z\right)\right)
\]
as $z\rightarrow\infty$. We say that  $\sum_{k=0}^{\infty}a_{k}\phi_{k}\left(z\right)$ 
is an asymptotic expansion for the function $f\left(z\right)$ as $z\rightarrow\infty$, 
and write 
\[
f\left(z\right)\sim\sum_{k=0}^{\infty}a_{k}\phi_{k}\left(z\right),
\]
if for all $N \in \mathbb N$
\[
f\left(z\right)=\sum_{k=0}^{N}a_{k}\phi_{k}\left(z\right)+o\left(\phi_{N}\left(z\right)\right)
\]
as $z\rightarrow\infty$ (cf., e.g., \cite{murray1992asymptotic}).

Let us write $\xi=\xi\left(u\right)$ for the non-zero
solution of 
\begin{equation}
\exp\left(\xi\right)=1+u\xi
\label{eq:DefXi}
\end{equation}
with $\xi\left(1\right)=0$. Note that $\log\left(u\right)<\xi\left(u\right)\leq2\log\left(u\right)$
for $u>1$.

\begin{thm} \label{quantitative}
	Assume that there are $\alpha_0, \alpha_1 \in (0,1)$ with 
	\[	
	\liminf_{n\to\infty} n^{-\alpha_0} \alpha(n) = \infty \quad \text{and} \quad 
	\limsup_{n\to\infty} n^{-\alpha_1}
	\alpha(n) = 0.
	\]
	Then the quantities defined in \eqref{mean and variance} have the asymptotic expansions
	\[
	m_{n,\alpha} \sim\frac{n}{\alpha\left(n\right)}\sum_{k=0}^{\infty}	
	\frac{k!}{\left(\xi\left(\frac{n}{\alpha\left(n\right)}\right)\right)^{k}}
	\]
	and
	\[
	v_{n,\alpha}\sim\frac{n}
	{\alpha\left(n\right)}\sum_{k=2}^{\infty}\frac{k!-1}
	{\left(\xi\left(\frac{n}{\alpha\left(n
	\right)}\right)\right)^{k}}.
	\]
	\end{thm}

This result shows that the average cycle count is close 
to its theoretical minimum, or in other words the average cycle lengths 
are close to maximal. This supports observations made e.g. in 
\cite{BUV} that 'entropy' is extremely strong in random permutations in the sense that
it is very hard to 'force' many short cycles by modifying the uniform measure on 
permutations in natural ways, such as cycle weights or nontrivial conditioning.

\section{Proofs: overview}

We will use the following criterion for a central limit theorem:

\begin{prop}
\label{prop:Sach}\cite{Sachkov1997} Let $A_n$ be a sequence of 
integer-valued random variables and assume that, as $n\rightarrow\infty$,
their probability generating functions $H_{n}$ can be written as 
\[
H_{n}\left(w\right)=\exp\left(h_{n}\left(w\right)\right)\left(1+o\left(1\right)\right),
\]
uniformly in the interval $W:=\left(1-\delta,1+\delta\right)$ for some $\delta > 0$, where $h_n$ has 
the following properties: $h_{n}$ is thrice continuously differentiable in $W$, and 
\begin{equation} \label{sachkov condition}
\lim_{n\to\infty} \sup_{w \in W} \frac{h_{n}^{\prime\prime\prime}\left(w\right)}{\left(h_{n}^{\prime}\left(1\right)+h_{n}^{\prime\prime}\left(1\right)\right)^{\frac{3}{2}}} = 0.
\end{equation}
Then the distribution of the random variable
\[
\eta_{n}:=\frac{A_{n}-h_{n}^{\prime}\left(1\right)}{\left(h_{n}^{\prime}\left(1\right)+h_{n}^{\prime\prime}\left(1\right)\right)^{\frac{1}{2}}}
\]
converges weakly to a standard normal distribution 
as $n\rightarrow\infty$.
\end{prop}

For applying the above result, we need information on the generating function of the 
random cycle count $C$ under $\mathbb P_{n,\alpha}$. As is well known (see e.g.\ \cite{Flajolet2009}), the cardinality of $S_{n,\alpha}$ is given by 
\begin{equation}
\left|S_{n,\alpha}\right|=n!\left[z^{n}\right]\exp\left(\sum_{j=1}^{\alpha\left(n\right)}\frac{z^{j}}{j}\right)=n!\frac{1}{2\pi i}\int_{\gamma}\exp\left(\sum_{j=1}^{\alpha\left(n\right)}\frac{z^{j}}{j}\right)\frac{\mathrm{d}z}{z^{n+1}}\label{eq:cardi}
\end{equation}
where $\gamma$ is any closed curve around the origin in the complex
plane. The probability generating function $H_{n,\alpha}\left(w\right)$
of $C$ under $\mathbb P_{n,\alpha}$ can therefore be written as
\begin{equation}
H_{n,\alpha}\left(w\right)=c_{n,\alpha}\frac{1}{2\pi i}\int_{\gamma}\exp\left(w\sum_{j=1}^{\alpha\left(n\right)}\frac{z^{j}}{j}\right)\frac{\mathrm{d}z}{z^{n+1}}\label{eq:GenFunc}
\end{equation}
where $c_{n,\alpha}$ is a normalization constant (which does not
depend on $w$) (cf. \cite{Flajolet2009} for both equations).

In \cite{Sachkov1997}, Proposition \ref{prop:Sach} is applied
to cases such as Equation (\ref{eq:GenFunc}), and an approximation using
the saddle point method is given. The difference to our situation however is that 
in the cited reference, there are fixed polynomials in the exponential, while in our case 
the upper limit $\alpha(n)$ of the sum will typically grow with $n$. To handle this 
situation, we use a very slight generalization of precise asymptotic estimates 
recently derived by Manstavi\v{c}ius
and Petuchovas \cite{manstavivcius2014local,manstavivcius2015permutations}.

\begin{prop}
\label{prop:asymp}
Under the assumptions of Theorem \ref{main thm},  the following asymptotic relation holds as $n\rightarrow\infty$,
uniformly for $w\in\left(1-\delta,1+\delta\right)$:
\[
H_{n,\alpha}\left(w\right)=c_{n,\alpha}\frac{\exp\left(w\sum_{j=1}^{\alpha\left(n\right)}\frac{x_{n,\alpha}^{j}\left(w\right)}{j}\right)}{x_{n,\alpha}^{n}\left(w\right)\sqrt{2\pi w\sum_{j=1}^{\alpha\left(n\right)}jx_{n,\alpha}^{j}\left(w\right)}}\left(1+O\left(\frac{\alpha\left(n\right)}{n}\right)\right).
\]
Here, $x_{n,\alpha}\left(w\right)$ is the unique positive solution of \eqref{eq:saddle}. 
\end{prop}
\begin{proof}
First note that because of the strict inequality in \eqref{growth condition}, there exists
$\delta > 0$ such that for all $w \in (1-\delta,1+\delta)$ and all large enough $n$, we have 
\[
\alpha\left(n\right)\leq \frac{1}{12\pi^{2}\mathrm{e}} \frac{n}{w\log\left(n\right)\left(\log\left(\log\left(n\right)\right)\right)^{2}}.
\]
The proposition is then an adaptation of Theorem 2 in \cite{manstavivcius2014local}.
The latter makes use of the saddle point method and treats the case
$w=1$. The proof given by Manstavi\v{c}ius and Petuchovas extends
almost verbatim to the present situation 
when only large $n$ are considered. The main difference is a maximization
of the resulting constants in the error terms with respect to $w\in\left(1-\delta,1+\delta\right)$,
which does not affect their asymptotics.
\end{proof}

\begin{rem}
Indeed, \cite{manstavivcius2014local} proves the asymptotics in the case $w=1$ without restrictions on $\alpha$,
but the generalization to $w \neq 1$ is not as obvious without the growth restriction \eqref{growth condition}, so we stick with it. 
\end{rem}

To connect Propositions \ref{prop:Sach} and \ref{prop:asymp} and
prove the theorem, we thus need asymptotics of the derivatives of 
\begin{equation}
h_{n,\alpha}\left(w\right)=\log\left(c_{n,\alpha}\right)+w\sum_{j=1}^{\alpha\left(n\right)}\frac{x_{n,\alpha}^{j}\left(w\right)}{j}-n\log\left(x_{n,\alpha}\left(w\right)\right)-\frac{1}{2}\log\left(2\pi w\sum_{j=1}^{\alpha\left(n\right)}jx_{n,\alpha}^{j}\left(w\right)\right).\label{eq:kleinH}
\end{equation}

This is done in Lemmas \ref{lem:h1Ableitung} to \ref{lem:hStrichStrichStrichAs}, 
where it is proved that 
\[
h_{n,\alpha}^{\prime}\left(w\right)= \sum_{j=1}^{\alpha\left(n\right)}\frac{x_{n,\alpha}^{j}\left(w\right)}{j}+O\left(1\right), \qquad 
h_{n,\alpha}^{\prime\prime}\left(w\right)= \frac{nx_{n,\alpha}^{\prime}\left(w\right)}{wx_{n,\alpha}\left(w\right)}+O\left(1\right),
\]
and that $h_{n,\alpha}'''(w) = O(n/\alpha(n))$. Since it turns out 
(see Lemma \ref{h2strich asymptotics} and the discussion thereafter) that the 
leading terms of both $h_{n,\alpha}'$ and $h_{n,\alpha}''$ behave like $n/\alpha(n)$ 
for large $n$ but have different signs, a very fine asymptotic analysis is necessary 
to ensure that \eqref{sachkov condition} holds. This is done in Propositions 
\ref{prop:hStrich} and \ref{prop:hStrichStrich}. After this, Proposition \ref{prop:Sach} 
can be used to prove Theorem \ref{main thm}. 

The proof of Theorem \ref{quantitative} refines the results of Propositions
\ref{prop:hStrich} and \ref{prop:hStrichStrich} by adding information on the 
asymptotics of $x_{n,\alpha}(1)$. These are given in Lemma \ref{lem:xiEigensch}, 
and are essentially due to \cite{manstavivcius2014local}.

\section{Proofs: details}

The following lemma presents an expression for the first derivative
of $h_{n,\alpha}$.
\begin{lem}
\label{lem:h1Ableitung}Let $\alpha:\mathbb{N}\rightarrow\mathbb{N}$
and $0<\delta<1$. Then the following relation holds for all $n\in\mathbb{N}$
and $w\in\left(1-\delta,1+\delta\right)$:
\begin{equation}
h_{n,\alpha}^{\prime}\left(w\right)=\sum_{j=1}^{\alpha\left(n\right)}\frac{x_{n,\alpha}^{j}\left(w\right)}{j}-\frac{1}{2}\frac{x_{n,\alpha}^{\prime}\left(w\right)}{x_{n,\alpha}\left(w\right)}+\frac{1}{2}\frac{x_{n,\alpha}^{\prime\prime}\left(w\right)}{x_{n,\alpha}^{\prime}\left(w\right)}+\frac{1}{2w}.\label{eq:kleinHstrich}
\end{equation}
\end{lem}
\begin{proof}
Since $x\mapsto\sum_{j=1}^{\alpha\left(n\right)}x^{j}$ is strictly
increasing on $\mathbb{R}_{+}$, the saddle point function $x_{n,\alpha}$
is infinitely often differentiable in $w$ by the inverse mapping
theorem and a bootstrapping argument. Due to Equation (\ref{eq:saddle}),
the derivative of the first three terms of Equation (\ref{eq:kleinH})
is the first term of Equation (\ref{eq:kleinHstrich}). For the last
term of Equation (\ref{eq:kleinH}), use the identity
\begin{equation}
\sum_{j=1}^{\alpha\left(n\right)}j\left(x_{n,\alpha}\left(w\right)\right)^{j}=-n\frac{x_{n,\alpha}\left(w\right)}{x_{n,\alpha}^{\prime}\left(w\right)}\frac{1}{w^{2}}\label{eq:XstrichUndJ1}
\end{equation}
obtained by differentiating Equation (\ref{eq:saddle}).
\end{proof}
The ingredients necessary for the proof of the theorem are asymptotics
for the derivatives of $h_{n,\alpha}$. A first step towards this
goal is
\begin{lem}
\label{lem:Divergenz}Let $\alpha:\mathbb{N}\rightarrow\mathbb{N}$
and $0<\delta<1$ such that $\frac{n}{w\alpha\left(n\right)}>1$ for
large $n$ and $w\in\left(1-\delta,1+\delta\right)$. Then
\[
\lim_{n\rightarrow\infty}\alpha\left(n\right)\log\left(x_{n,\alpha}\left(w\right)\right)=\infty
\]
if and only if
\[
\lim_{n\rightarrow\infty}\frac{\alpha\left(n\right)}{n}=0.
\]
\end{lem}
\begin{proof}
\cite{manstavivcius2014local} Comparing geometric and arithmetic
means yields
\begin{equation}
\left(x_{n,\alpha}\left(w\right)\right)^{\frac{\alpha\left(n\right)+1}{2}}=\left(\prod_{j=1}^{\alpha\left(n\right)}\left(x_{n,\alpha}\left(w\right)\right)^{j}\right)^{\frac{1}{\alpha\left(n\right)}}\leq\frac{1}{\alpha\left(n\right)}\sum_{j=1}^{\alpha\left(n\right)}\left(x_{n,\alpha}\left(w\right)\right)^{j}=\frac{n}{w\alpha\left(n\right)}.\label{eq:GeomArith}
\end{equation}
From Equation (\ref{eq:saddle}),
\[
\left(x_{n,\alpha}\left(w\right)\right)^{\alpha\left(n\right)}\geq\frac{n}{w\alpha\left(n\right)}.
\]
Hence,
\[
\left(x_{n,\alpha}\left(w\right)\right)^{\frac{\alpha\left(n\right)+1}{2}}\leq\frac{n}{w\alpha\left(n\right)}\leq\left(x_{n,\alpha}\left(w\right)\right)^{\alpha\left(n\right)}
\]
and, since the logarithm is increasing and $\log\left(x_{n,\alpha}\left(w\right)\right)\geq0$,
\begin{equation}
\frac{1}{2}\alpha\left(n\right)\log\left(x_{n,\alpha}\left(w\right)\right)\leq\log\left(\frac{n}{w\alpha\left(n\right)}\right)\leq\alpha\left(n\right)\log\left(x_{n,\alpha}\left(w\right)\right).\label{eq:alphaLogX}
\end{equation}
The claim is a direct consequence.\end{proof}
\begin{rem}
More detailed asymptotics for $x_{n,\alpha}\left(w\right)$ and related
quantities can be found in \cite{manstavivcius2014local}.\end{rem}
\begin{lem}
\label{lem:asympXstrich}Let $\alpha:\mathbb{N}\rightarrow\mathbb{N}$
and $0<\delta<1$ such that $\frac{n}{w\alpha\left(n\right)}>1$ for
large $n$ and $w\in\left(1-\delta,1+\delta\right)$. Moreover, let
\[
\lim_{n\rightarrow\infty}\frac{\alpha\left(n\right)}{n}=0.
\]
Then,
\[
\lim_{n\rightarrow\infty}\alpha\left(n\right)\frac{x_{n,\alpha}^{\prime}\left(w\right)}{x_{n,\alpha}\left(w\right)}=-\frac{1}{w}.
\]
\end{lem}
\begin{proof}
Starting from Equation (\ref{eq:saddle}), we have
\begin{equation}
w\left[\left(x_{n,\alpha}\left(w\right)\right)^{\alpha\left(n\right)}-1\right]=n\left(1-\frac{1}{x_{n,\alpha}\left(w\right)}\right).\label{eq:ZuDiff}
\end{equation}
Differentiating Equation (\ref{eq:ZuDiff}) with respect to $w$,
solving for $x_{n,\alpha}^{\prime}\left(w\right)$, and applying once
more Equation (\ref{eq:saddle}) yields
\begin{equation}
x_{n,\alpha}^{\prime}\left(w\right)=\frac{\left(x_{n,\alpha}\left(w\right)-1\right)x_{n,\alpha}\left(w\right)}{w\left(1-\alpha\left(n\right)\left(x_{n,\alpha}\left(w\right)-1\right)-\frac{w\alpha\left(n\right)x_{n,\alpha}\left(w\right)}{n}\right)}.\label{eq:XstrichfuerProp}
\end{equation}
Hence,
\begin{align*}
\lim_{n\rightarrow\infty}\frac{x_{n,\text{\ensuremath{\alpha}}}\left(w\right)}{\alpha\left(n\right)x_{n,\alpha}^{\prime}\left(w\right)}= & \lim_{n\rightarrow\infty}\left[\frac{w}{\alpha\left(n\right)\left(x_{n,\alpha}\left(w\right)-1\right)}-w-\frac{w^{2}x_{n,\alpha}\left(w\right)}{n\left(x_{n,\alpha}\left(w\right)-1\right)}\right]\\
= & -w
\end{align*}
by Lemma \ref{lem:Divergenz} and
\[
x_{n,\alpha}\left(w\right)-1\geq\log\left(x_{n,\alpha}\left(w\right)\right).
\]
\end{proof}

\begin{lem}
\label{lem:hAbleitungen}Let $\alpha:\mathbb{N}\rightarrow\mathbb{N}$
and $0<\delta<1$ as in Lemma \ref{lem:asympXstrich}. Then the following
relations hold uniformly in $w\in\left(1-\delta,1+\delta\right)$
as $n\rightarrow\infty$:
\begin{align}
h_{n,\alpha}^{\prime}\left(w\right)= & \sum_{j=1}^{\alpha\left(n\right)}\frac{x_{n,\alpha}^{j}\left(w\right)}{j}+O\left(1\right),\label{eq:11-1}\\
h_{n,\alpha}^{\prime\prime}\left(w\right)= & \frac{nx_{n,\alpha}^{\prime}\left(w\right)}{wx_{n,\alpha}\left(w\right)}+O\left(1\right),\label{eq:11-2}
\end{align}
and
\begin{align}
h_{n,\alpha}^{\prime\prime\prime}\left(w\right)= & \frac{nx_{n,\alpha}^{\prime\prime}\left(w\right)}{wx_{n,\alpha}\left(w\right)}-\frac{nx_{n,\alpha}^{\prime}\left(w\right)}{w^{2}x_{n,\alpha}\left(w\right)}-\frac{n}{w}\left(\frac{x_{n,\alpha}^{\prime}\left(w\right)}{x_{n,\alpha}\left(w\right)}\right)^{2}+O\left(1\right).\label{eq:11-3}
\end{align}
\end{lem}
\begin{proof}
Let
\[
J_{k}\left(x\right):=\sum_{j=1}^{\alpha\left(n\right)}j^{k}x^{k},
\]
then
\[
\lim_{n\rightarrow\infty}\frac{J_{1}\left(x_{n,\alpha}\left(w\right)\right)}{n\alpha\left(n\right)}=\frac{1}{w}
\]
by Lemma \ref{lem:asympXstrich} and
\begin{equation}
\frac{x_{n,\alpha}^{\prime}\left(w\right)}{x_{n,\alpha}\left(w\right)}J_{1}\left(x_{n,\alpha}\left(w\right)\right)=-\frac{n}{w^{2}},\label{eq:xStrichplusJ}
\end{equation}
which is a consequence of Equation (\ref{eq:XstrichUndJ1}). Moreover,
\[
J_{k}\left(x_{n,\alpha}\left(w\right)\right)\leq\left(\alpha\left(n\right)\right)^{k-1}J_{1}\left(x_{n,\alpha}\left(w\right)\right)
\]
for all $k\geq1$. Differentiating Equation (\ref{eq:XstrichUndJ1})
yields
\begin{align*}
\frac{x_{n,\alpha}^{\prime\prime}\left(w\right)}{x_{n,\alpha}\left(w\right)}J_{1}\left(x_{n,\alpha}\left(w\right)\right)+\frac{n}{w^{2}}\frac{x_{n,\alpha}^{\prime}\left(w\right)}{x_{n,\alpha}\left(w\right)}+\left(\frac{x_{n,\alpha}^{\prime}\left(w\right)}{x_{n,\alpha}\left(w\right)}\right)^{2}J_{2}\left(x_{n,\alpha}\left(w\right)\right) & =\frac{2n}{w^{3}}.
\end{align*}
Hence,
\[
\frac{x_{n,\alpha}^{\prime\prime}\left(w\right)}{x_{n,\alpha}\left(w\right)}=O\left(\frac{1}{\alpha\left(n\right)}\right)
\]
uniformly in $w\in\left(1-\delta,1+\delta\right)$ by Lemma \ref{lem:asympXstrich}. Iterating this
procedure, we obtain
\[
\frac{x_{n,\alpha}^{\prime\prime\prime}\left(w\right)}{x_{n,\alpha}\left(w\right)}=O\left(\frac{1}{\alpha\left(n\right)}\right)
\]
and
\[
\frac{x_{n,\alpha}^{\left(4\right)}\left(w\right)}{x_{n,\alpha}\left(w\right)}=O\left(\frac{1}{\alpha\left(n\right)}\right).
\]
Equation (\ref{eq:11-1}) is now a direct consequence. Since it is
easily shown that the derivatives of
\[
\frac{x_{n,\alpha}^{\prime}\left(w\right)}{x_{n,\alpha}\left(w\right)}
\]
and
\[
\frac{x_{n,\alpha}^{\prime\prime}\left(w\right)}{x_{n,\alpha}^{\prime}\left(w\right)}
\]
are also $O\left(1\right)$, Equation (\ref{eq:11-2}) follows from
Equations (\ref{eq:11-1}) and (\ref{eq:saddle}). Further differentiation
yields Equation (\ref{eq:11-3}).\end{proof}
\begin{lem}
\label{lem:hStrichStrichStrichAs}Under the assumptions of Lemma \ref{lem:asympXstrich},
\[
h_{n,\alpha}^{\prime\prime\prime}\left(w\right)=O\left(\frac{n}{\alpha\left(n\right)}\right)
\]
uniformly in $w\in\left(1-\delta,1+\delta\right)$ as $n\rightarrow\infty$.\end{lem}
\begin{proof}
After the work done in the previous lemmata, this is an easy corollary.
The dominating terms in Equation (\ref{eq:11-3}) have the required
property.\end{proof}
\begin{lem} \label{h2strich asymptotics}
Under the assumptions of Lemma \ref{lem:asympXstrich},
\[
\lim_{n\rightarrow\infty}\frac{\alpha\left(n\right)}{n}h_{n,\alpha}^{\prime\prime}\left(1\right)=-1.
\]
\end{lem}
\begin{proof}
The lemma follows from Equation (\ref{eq:11-2}) and Lemma \ref{lem:asympXstrich}.
\end{proof}

The asymptotic behaviour of $h_{n,\alpha}^{\prime}\left(1\right)$
is determined by
\begin{align*}
\sum_{j=1}^{\alpha\left(n\right)}\frac{x_{n,\alpha}^{j}\left(1\right)}{j}\geq & \frac{1}{\alpha\left(n\right)}\sum_{j=1}^{\alpha\left(n\right)}x_{n,\alpha}^{j}\left(1\right)\\
= & \frac{n}{\alpha\left(n\right)}
\end{align*}
which competes with $-\frac{n}{\alpha\left(n\right)}$ due to the
different sign. If the third derivative of $h_{n,\alpha}$ is to be
dominated by
\[
\left(h_{n,\alpha}^{\prime}\left(1\right)+h_{n,\alpha}^{\prime\prime}\left(1\right)\right)^{\frac{3}{2}},
\]
more careful estimates are needed. Such estimates will be provided
by Propositions \ref{prop:hStrich} and \ref{prop:hStrichStrich}.

The exponential integral (\cite{NIST:DLMF}) is given by 
\[
\mathrm{Ei}\left(x\right):=\mathrm{p.v.}\int_{-\infty}^{x}\frac{\exp\left(t\right)}{t}\mathrm{d}t
\]
and has the asymptotic expansion
\begin{equation}
\mathrm{Ei}\left(x\right)\sim\frac{\exp\left(x\right)}{x}\sum_{k=0}^{\infty}\frac{k!}{x^{k}}\label{eq:ExpIntAsymp}
\end{equation}
as $x\rightarrow\infty$.

\begin{prop}
\label{prop:hStrich}Let $\alpha:\mathbb{N}\rightarrow\mathbb{N}$
such that
\[
\lim_{n\rightarrow\infty}\frac{\alpha\left(n\right)}{n}=0
\]
and $\alpha\left(n\right)\geq2$ for large $n$. Then
\begin{align}
 & h_{n,\alpha}^{\prime}\left(1\right)\nonumber \\
= & \log\left(\alpha\left(n\right)\right)+\int_{0}^{\log\left(x_{n,\alpha}\left(1\right)\right)}\frac{\exp\left(\alpha\left(n\right)v\right)-1}{v}\frac{v\exp\left(v\right)\mathrm{d}v}{\exp\left(v\right)-1}+O\left(1\right)\nonumber \\
= & \log\left(\alpha\left(n\right)\right)+\frac{n}{\alpha\left(n\right)}+\frac{n}{\left(\alpha\left(n\right)\right)^{2}\left(x_{n,\alpha}\left(1\right)-1\right)}+2\frac{n}{\left(\alpha\left(n\right)\right)^{3}\left(x_{n,\alpha}\left(1\right)-1\right)\log\left(x_{n,\alpha}\left(1\right)\right)}\label{eq:Erwartungswert}\\
 & +p_{\alpha}\left(n\right)+O\left(\frac{n}{\left(\alpha\left(n\right)\right)^{4}\left(\log\left(x_{n,\alpha}\left(1\right)\right)\right)^{2}\left(x_{n,\alpha}\left(1\right)-1\right)}\right).\nonumber 
\end{align}
Here, $p_{\alpha}$ is a non-negative function satisfying $p_{\alpha}(n)=O\left( \frac{\left( x_{n,\alpha}(1) \right)^{\alpha(n)}}{(\alpha (n))^{3} \log (x_{n,\alpha}(1))} \right)$.\end{prop}
\begin{proof}
Lemma \ref{lem:hAbleitungen} shows that only one term is relevant
for the proof of the proposition.
\begin{align}
\sum_{j=1}^{\alpha\left(n\right)}\frac{x_{n,\alpha}^{j}\left(1\right)}{j}= & \sum_{j=1}^{\alpha\left(n\right)}\frac{1}{j}+\int_{1}^{x_{n,\alpha}\left(1\right)}\sum_{j=1}^{\alpha\left(n\right)}t^{j-1}\mathrm{d}t\nonumber \\
= & \sum_{j=1}^{\alpha\left(n\right)}\frac{1}{j}+\int_{0}^{\log\left(x_{n,\alpha}\left(1\right)\right)}\frac{\exp\left(\alpha\left(n\right)v\right)-1}{v}\frac{v\exp\left(v\right)\mathrm{d}v}{\exp\left(v\right)-1}.\label{eq:SumToInt}
\end{align}
The second line follows from
\begin{equation}
\sum_{j=1}^{\alpha\left(n\right)}t^{j-1}=\frac{t^{\alpha\left(n\right)}-1}{t-1}\label{eq:GeomReihe}
\end{equation}
and the substitution $v=\log\left(t\right)$ (\cite{manstavivcius2014local}).
It is a well-known fact about the harmonic numbers that
\[
\sum_{j=1}^{\alpha\left(n\right)}\frac{1}{j}=\log\left(\alpha\left(n\right)\right)+O\left(1\right).
\]
Let
\[
g\left(v\right):=\frac{v\exp\left(v\right)}{\exp\left(v\right)-1}.
\]
Then
\begin{align*}
g^{\prime}\left(v\right)= & \frac{\exp\left(v\right)\left[\exp\left(v\right)-1-v\right]}{\left(\exp\left(v\right)-1\right)^{2}}
\end{align*}
and
\[
g^{\prime\prime}\left(v\right)=\frac{\exp\left(v\right)\left[v\exp\left(v\right)-2\exp\left(v\right)+v+2\right]}{\left(\exp\left(v\right)-1\right)^{3}}.
\]
Expanding the exponential functions in the numerator of $g^{\prime\prime}$
shows that
\[
g^{\prime\prime}\left(v\right)\geq0
\]
for all $v\geq0$. Moreover, $g^{\prime\prime}$ is bounded on $\mathbb{R}_{+}$ since $\lim_{v \to 0}g^{\prime\prime}(v) = \frac{1}{6}$ and $\lim_{v \to \infty}g^{\prime\prime}(v) = 0$.

By expanding $g$ about $\log\left(x_{n,\alpha}\left(1\right)\right)$,
the integrand in Equation (\ref{eq:SumToInt}) becomes
\begin{align*}
 & \frac{\exp\left(\alpha\left(n\right)v\right)-1}{v}\frac{v\exp\left(v\right)}{\exp\left(v\right)-1}\\
= & \frac{\log\left(x_{n,\alpha}\left(1\right)\right)}{1-\left(x_{n,\alpha}\left(1\right)\right)^{-1}}\frac{\exp\left(\alpha\left(n\right)v\right)-1}{v}\\
 & +\frac{\left(x_{n,\alpha}\left(1\right)-1\right)-\log\left(x_{n,\alpha}\left(1\right)\right)}{\left(x_{n,\alpha}\left(1\right)-1\right)\left(1-\left(x_{n,\alpha}\left(1\right)\right)^{-1}\right)}\frac{\exp\left(\alpha\left(n\right)v\right)-1}{v}\left(v-\log\left(x_{n,\alpha}\left(1\right)\right)\right)\\
 & +\frac{g^{\prime\prime}\left(\Xi\left(v\right)\right)}{2}\frac{\exp\left(\alpha\left(n\right)v\right)-1}{v}\left(v-\log\left(x_{n,\alpha}\left(1\right)\right)\right)^{2}\\
= & \frac{1}{\left(1-\left(x_{n,\alpha}\left(1\right)\right)^{-1}\right)}\left[1-\frac{\log\left(x_{n,\alpha}\left(1\right)\right)}{x_{n,\alpha}\left(1\right)-1}\right]\left(\exp\left(\alpha\left(n\right)v\right)-1\right)\\
 & +\frac{\left(\log\left(x_{n,\alpha}\left(1\right)\right)\right)^{2}}{\left(x_{n,\alpha}\left(1\right)-1\right)\left(1-\left(x_{n,\alpha}\left(1\right)\right)^{-1}\right)}\frac{\exp\left(\alpha\left(n\right)v\right)-1}{v}\\
 & +\frac{g^{\prime\prime}\left(\Xi\left(v\right)\right)}{2}\frac{\exp\left(\alpha\left(n\right)v\right)-1}{v}\left(v-\log\left(x_{n,\alpha}\left(1\right)\right)\right)^{2}
\end{align*}
where $0\leq\Xi\left(v\right)\leq v$ by Taylor's theorem. An easy
calculation yields
\begin{align}
 & \int_{0}^{\log\left(x_{n,\alpha}\left(1\right)\right)}\left(\exp\left(\alpha\left(n\right)v\right)-1\right)\mathrm{d}v\nonumber \\
= & \frac{\left(x_{n,\alpha}\left(1\right)\right)^{\alpha\left(n\right)}}{\alpha\left(n\right)}-\log\left(x_{n,\alpha}\left(1\right)\right)-\frac{1}{\alpha\left(n\right)}\label{eq:Integral0}
\end{align}
and, by substituting $s=\frac{v}{\alpha\left(n\right)}$ and applying
Equation (\ref{eq:ExpIntAsymp}),
\begin{align}
 & \int_{0}^{\log\left(x_{n,\alpha}\left(1\right)\right)}\frac{\exp\left(\alpha\left(n\right)v\right)-1}{v}\mathrm{d}v\nonumber \\
= & \mathrm{Ei}\left(\alpha\left(n\right)\log\left(x_{n,\alpha}\left(1\right)\right)\right)-\log\log\left(\left(x_{n,\alpha}\left(1\right)\right)^{\alpha\left(n\right)}\right)+O\left(1\right)\nonumber \\
= & \frac{\left(x_{n,\alpha}\left(1\right)\right)^{\alpha\left(n\right)}}{\alpha\left(n\right)\log\left(x_{n,\alpha}\left(1\right)\right)}\left[1+\frac{1}{\alpha\left(n\right)\log\left(x_{n,\alpha}\left(1\right)\right)}+\frac{2}{\left(\alpha\left(n\right)\log\left(x_{n,\alpha}\left(1\right)\right)^{2}\right)}\right]\nonumber \\
 & +O\left(\frac{\left(x_{n,\alpha}\left(1\right)\right)^{\alpha\left(n\right)}}{\alpha\left(n\right)\log\left(x_{n,\alpha}\left(1\right)\right)}\frac{1}{\left(\alpha\left(n\right)\log\left(x_{n,\alpha}\left(1\right)\right)\right)^{3}}\right).\label{eq:ExpIntegral}
\end{align}
Due to Equations (\ref{eq:saddle}) and (\ref{eq:GeomReihe}),
\begin{align}
\left(x_{n,\alpha}\left(1\right)\right)^{\alpha\left(n\right)}= & 1+n\left(1-\left(x_{n,\alpha}\left(1\right)\right)^{-1}\right)\label{eq:SattelExakt}\\
\sim &  n\left(1-\left(x_{n,\alpha}\left(1\right)\right)^{-1}\right).\nonumber 
\end{align}
Applying Equations (\ref{eq:Integral0}), (\ref{eq:ExpIntegral}),
and (\ref{eq:SattelExakt}) to the integral in Equation (\ref{eq:SumToInt}),
we have
\begin{align*}
 & \int_{0}^{\log\left(x_{n,\alpha}\left(1\right)\right)}\frac{\exp\left(\alpha\left(n\right)v\right)-1}{v}\frac{v\exp\left(v\right)}{\exp\left(v\right)-1}\mathrm{d}v\\
= & \frac{n}{\alpha\left(n\right)}+\frac{n}{\left(\alpha\left(n\right)\right)^{2}\left(x_{n,\alpha}\left(1\right)-1\right)}+2\frac{n}{\left(\alpha\left(n\right)\right)^{3}\left(x_{n,\alpha}\left(1\right)-1\right)\log\left(x_{n,\alpha}\left(1\right)\right)}\\
 & +p_{\alpha}\left(n\right)+O\left(\frac{n}{\left(\alpha\left(n\right)\right)^{4}\left(\log\left(x_{n,\alpha}\left(1\right)\right)\right)^{2}\left(x_{n,\alpha}\left(1\right)-1\right)}\right)
\end{align*}
with
\[
p_{\alpha}\left(n\right):=\int_{0}^{\log\left(x_{n,\alpha}\left(1\right)\right)}\frac{g^{\prime\prime}\left(\Xi\left(v\right)\right)}{2}\frac{\exp\left(\alpha\left(n\right)v\right)-1}{v}\left(v-\log\left(x_{n,\alpha}\left(1\right)\right)\right)^{2}\mathrm{d}v \geq 0.
\]

Since $g^{\prime\prime}$ is bounded, a calculation similar to Equations (\ref{eq:Integral0}) and (\ref{eq:ExpIntegral}) yields
\[
p_{\alpha}\left( n \right) = O\left( \frac{\left( x_{n,\alpha}(1) \right)^{\alpha(n)}}{(\alpha (n))^{3} \log (x_{n,\alpha}(1))} \right).
\]

Finally, since by Equation (\ref{eq:alphaLogX})
\begin{align}
 & \alpha\left(n\right)\left(x_{n,\alpha}\left(1\right)-1\right)\nonumber \\
= & \alpha\left(n\right)\left[\exp\left(\log\left(x_{n,\alpha}\left(1\right)\right)\right)-1\right]\nonumber \\
\leq & \alpha\left(n\right)\left[\exp\left(2\frac{\log\left(\frac{n}{\alpha\left(n\right)}\right)}{\alpha\left(n\right)}\right)-1\right]\nonumber \\
= & O\left(\log\left(\frac{n}{\alpha\left(n\right)}\right)\right)\label{eq:alphaXgro}
\end{align}
for $\alpha\left(n\right)\geq\log\left(n\right)\geq\log\left(\frac{n}{\alpha\left(n\right)}\right)$
and
\begin{equation}
\alpha\left(n\right)\left(x_{n,\alpha}\left(1\right)-1\right)\leq\alpha\left(n\right)n^{\frac{1}{\alpha\left(n\right)}}\label{eq:alphaXklei}
\end{equation}
by Equation (\ref{eq:saddle}), we have
\[
1=O\left(\frac{n}{\left(\alpha\left(n\right)\right)^{4}\left(\log\left(x_{n,\alpha}\left(1\right)\right)\right)^{2}\left(x_{n,\alpha}\left(1\right)-1\right)}\right)
\]
and the claim is proved.\end{proof}
\begin{prop}
\label{prop:hStrichStrich}Let $\alpha:\mathbb{N}\rightarrow\mathbb{N}$
such that
\[
\lim_{n\rightarrow\infty}\frac{\alpha\left(n\right)}{n}=0.
\]
Then
\[
h_{n,\alpha}^{\prime\prime}\left(1\right)=-\frac{n}{\alpha\left(n\right)}\sum_{j=0}^{\infty}\left(\frac{1}{\alpha\left(n\right)\left(x_{n,\alpha}\left(1\right)-1\right)+\frac{\alpha\left(n\right)x_{n,\alpha}\left(1\right)}{n}}\right)^{j}+O\left(1\right)
\]
as $n\rightarrow\infty$.

If also $\alpha\left(n\right)\geq2$ for large $n$, then
\[
h_{n,\alpha}^{\prime\prime}\left(1\right)=-\frac{n}{\alpha\left(n\right)}\sum_{j=0}^{2}\left(\frac{1}{\alpha\left(n\right)\left(x_{n,\alpha}\left(1\right)-1\right)}\right)^{j}+o\left(\frac{n}{\left(\alpha\left(n\right)\right)^{3}\left(x_{n,\alpha}\left(1\right)-1\right)^{2}}\right)+O\left(1\right).
\]
\end{prop}

\begin{proof}
According to Lemma \ref{lem:hAbleitungen} and Equation (\ref{eq:XstrichfuerProp}),
\begin{align*}
h_{n,\alpha}^{\prime\prime}\left(1\right)= & \frac{nx_{n,\alpha}^{\prime}\left(1\right)}{x_{n,\alpha}\left(1\right)}+O\left(1\right)\\
= & \frac{n\left(x_{n,\alpha}\left(1\right)-1\right)}{1-\alpha\left(n\right)\left(x_{n,\alpha}\left(1\right)-1\right)-\frac{\alpha\left(n\right)x_{n,\alpha}\left(1\right)}{n}}+O\left(1\right)\\
= & \frac{n}{\alpha\left(n\right)}\frac{1}{\frac{1}{\alpha\left(n\right)\left(x_{n,\alpha}\left(1\right)-1\right)+\frac{\alpha\left(n\right)x_{n,\alpha}\left(1\right)}{n}}-1}\\
 & -\frac{x_{n,\alpha}\left(1\right)}{1-\left(\alpha\left(n\right)\left(x_{n,\alpha}\left(1\right)-1\right)+\frac{\alpha\left(n\right)x_{n,\alpha}\left(1\right)}{n}\right)}+O\left(1\right)\\
= & -\frac{n}{\alpha\left(n\right)}\sum_{j=0}^{\infty}\left(\frac{1}{\alpha\left(n\right)\left(x_{n,\alpha}\left(1\right)-1\right)+\frac{\alpha\left(n\right)x_{n,\alpha}\left(1\right)}{n}}\right)^{j}+O\left(1\right).
\end{align*}
The last line applies the geometric series and
\[
\frac{x_{n,\alpha}\left(1\right)}{1-\left(\alpha\left(n\right)\left(x_{n,\alpha}\left(1\right)-1\right)+\frac{\alpha\left(n\right)x_{n,\alpha}\left(1\right)}{n}\right)}=O\left(1\right),
\]
which follows easily from Lemma \ref{lem:Divergenz}.

From
\begin{align*}
 & \frac{1}{\alpha\left(n\right)\left(x_{n,\alpha}\left(1\right)-1\right)+\frac{\alpha\left(n\right)x_{n,\alpha}\left(1\right)}{n}}\\
= & \frac{1}{\alpha\left(n\right)\left(x_{n,\alpha}\left(1\right)-1\right)}-\frac{\alpha\left(n\right)}{n}\frac{x_{n,\alpha}\left(1\right)}{\left(\alpha\left(n\right)\right)^{2}\left(x_{n,\alpha}\left(1\right)-1\right)^{2}}
\end{align*}
and
\[
x_{n,\alpha}\left(1\right)\leq\left(\frac{n}{\alpha\left(n\right)}\right)^{\frac{2}{1+\alpha\left(n\right)}}
\]
by Equation (\ref{eq:GeomArith}), one may conclude that
\[
\frac{\alpha\left(n\right)}{n}\frac{x_{n,\alpha}\left(1\right)}{\left(\alpha\left(n\right)\right)^{2}\left(x_{n,\alpha}\left(1\right)-1\right)^{2}}=o\left(\left(\frac{1}{\alpha\left(n\right)\left(x_{n,\alpha}\left(1\right)-1\right)}\right)^{2}\right)
\]
if $\alpha\left(n\right)\geq2$. The second claim then follows.
\end{proof}
The tools needed to prove Theorem \ref{main thm} are now available.
\begin{proof}[Proof of Theorem \ref{main thm}]
Propositions \ref{prop:hStrich} and \ref{prop:hStrichStrich} yield
\begin{align}
 & h_{n,\alpha}^{\prime}\left(1\right)+h_{n,\alpha}^{\prime\prime}\left(1\right)\label{eq:Varianz}\\
= & \log\left(\alpha\left(n\right)\right)+\frac{n}{\alpha\left(n\right)}+\frac{n}{\left(\alpha\left(n\right)\right)^{2}\left(x_{n,\alpha}\left(1\right)-1\right)}+2\frac{n}{\left(\alpha\left(n\right)\right)^{3}\left(x_{n,\alpha}\left(1\right)-1\right)\log\left(x_{n,\alpha}\left(1\right)\right)}\nonumber \\
 & +p_{\alpha}\left(n\right)+O\left(\frac{n}{\left(\alpha\left(n\right)\right)^{4}\left(\log\left(x_{n,\alpha}\left(1\right)\right)\right)^{2}\left(x_{n,\alpha}\left(1\right)-1\right)}\right)\nonumber \\
 & -\frac{n}{\alpha\left(n\right)}\sum_{j=0}^{2}\left(\frac{1}{\alpha\left(n\right)\left(x_{n,\alpha}\left(1\right)-1\right)}\right)^{j}+o\left(\frac{n}{\left(\alpha\left(n\right)\right)^{3}\left(x_{n,\alpha}\left(1\right)-1\right)^{2}}\right).\nonumber 
\end{align}
Since
\[
x-1\geq\log\left(x\right)
\]
for $x>1$, we have
\[
\lim_{n\rightarrow\infty}\frac{h_{n,\alpha}^{\prime}\left(1\right)+h_{n,\alpha}^{\prime\prime}\left(1\right)}{\frac{n}{\left(\alpha\left(n\right)\right)^{3}\left(x_{n,\alpha}\left(1\right)-1\right)\log\left(x_{n,\alpha}\left(1\right)\right)}}\geq1.
\]
By Lemma \ref{lem:hStrichStrichStrichAs} and Equations (\ref{eq:alphaLogX}),
(\ref{eq:alphaXgro}), and (\ref{eq:alphaXklei}) as well as $\alpha\left(n\right)\geq4$,
\[
\lim_{n\rightarrow\infty}\frac{h_{n,\alpha}^{\prime\prime\prime}\left(w\right)}{\left(\frac{n}{\left(\alpha\left(n\right)\right)^{3}\left(x_{n,\alpha}\left(1\right)-1\right)\log\left(x_{n,\alpha}\left(1\right)\right)}\right)^{\frac{3}{2}}}=0
\]
uniformly in $w$.

Therefore, Proposition \ref{prop:Sach} may be applied and the theorem
is proved.\end{proof}
\begin{rem}
Equations (\ref{eq:Erwartungswert}) and (\ref{eq:Varianz}) show
that the behavior of $m_{n,\alpha}$ and $v_{n,\alpha}$ changes when $\log\left(\alpha\left(n\right)\right)$
surpasses $\frac{n}{\alpha\left(n\right)}$ and becomes the dominating
term. This blends in nicely with the fact that the classical uniform
model has asymptotic expectation and variance of $\log\left(n\right)$.\end{rem}
\begin{lem}
\label{lem:xiEigensch}\cite{manstavivcius2014local} Let $\alpha:\mathbb{N}\rightarrow\mathbb{N}$
such that $\log\left(n\right)\leq\alpha\left(n\right)<n$ for all
$n$. Then
\[
x_{n,\alpha}\left(1\right)=\exp\left(\frac{\xi\left(\frac{n}{\alpha\left(n\right)}\right)}{\alpha\left(n\right)}\right)+O\left(\frac{\log\left(\frac{n}{\alpha\left(n\right)}+1\right)}{\left(\alpha\left(n\right)\right)^{2}}\right),
\]
\[
\xi\left(\frac{n}{\alpha\left(n\right)}\right)=\log\left(\frac{n}{\alpha\left(n\right)}\right)+\log\left(\log\left(\frac{n}{\alpha\left(n\right)}+2\right)\right)+O\left(\frac{\log\left(\log\left(\frac{n}{\alpha\left(n\right)}+2\right)\right)}{\log\left(\frac{n}{\alpha\left(n\right)}+2\right)}\right)
\]
and
\[
\alpha\left(n\right)\log\left(x_{n,\alpha}\left(1\right)\right)=\xi\left(\frac{n}{\alpha\left(n\right)}\right)+O\left(\frac{\log\left(\frac{n}{\alpha\left(n\right)}+1\right)}{\alpha\left(n\right)}\right)
\]
as $n\rightarrow\infty$.\end{lem}
\begin{proof}
The first and second parts are reformulations of assertions in Lemmata
5 and 1 in \cite{manstavivcius2014local} from which the third part
follows. The third equation is already stated in the proof of Corollary
1 in \cite{manstavivcius2014local}, albeit for a smaller range of
possible $\alpha$.\end{proof}
\begin{lem}
\label{lem:Tabschaetzung}\cite{manstavivcius2014local} Let 
\begin{equation}
T_{K}\left(z\right):=\int_{0}^{z}\frac{\exp\left(t\right)-1}{t}\left(\frac{t}{K}\frac{\exp\left(\frac{t}{K}\right)}{\exp\left(\frac{t}{K}\right)-1}-1\right)\mathrm{d}t\label{eq:defT}
\end{equation}
for $K>0.$ If $0\leq z\leq\pi K$, then
\[
\left|T_{K}\left(z\right)+\frac{z}{2K}\right|\leq\frac{4\exp\left(z\right)}{K}.
\]
\end{lem}
\begin{proof}
This is basically Lemma 6 in \cite{manstavivcius2014local}.
\end{proof}
Theorem \ref{quantitative} can now be proved.
\begin{proof}[Proof of Theorem \ref{quantitative}]
We only need to verify the asymptotics of $h_{n,\alpha}^{\prime}\left(1\right)$
and $h_{n,\alpha}^{\prime}\left(1\right)+h_{n,\alpha}^{\prime\prime}\left(1\right)$.
Concerning $h_{n,\alpha}^{\prime}\left(1\right)$, by Proposition
\ref{prop:hStrich} and a substitution, the important term is
\begin{align*}
 & \int_{0}^{\alpha\left(n\right)\log\left(x_{n,\alpha}\left(1\right)\right)}\frac{\exp\left(v\right)-1}{v}\frac{v}{\alpha\left(n\right)}\frac{\exp\left(\frac{v}{\alpha\left(n\right)}\right)\mathrm{d}v}{\exp\left(\frac{v}{\alpha\left(n\right)}\right)-1}\\
= & T_{\alpha\left(n\right)}\left(\alpha\left(n\right)\log\left(x_{n,\alpha}\left(1\right)\right)\right)+I\left(\alpha\left(n\right)\log\left(x_{n,\alpha}\left(1\right)\right)\right)
\end{align*}
where
\[
I\left(z\right):=\int_{0}^{z}\frac{\exp\left(t\right)-1}{t}\mathrm{d}t.
\]
Since $0\leq\alpha\left(n\right)\log\left(x_{n,\alpha}\left(1\right)\right)\leq\pi\alpha\left(n\right)$
for $n$ large enough ($\log\left(x_{n,\alpha}\left(1\right)\right)\rightarrow0$
for $n\rightarrow\infty$ by Lemma \ref{lem:xiEigensch}), we can
apply Lemma \ref{lem:Tabschaetzung}. The resulting terms 
\[
\frac{\alpha\left(n\right)\log\left(x_{n,\alpha}\left(1\right)\right)}{2\alpha\left(n\right)}=\frac{\log\left(x_{n,\alpha}\left(1\right)\right)}{2}\longrightarrow0
\]
and
\begin{align*}
\frac{4\exp\left(\alpha\left(n\right)\log\left(x_{n,\alpha}\left(1\right)\right)\right)}{\alpha\left(n\right)}= & O\left(\frac{\exp\left(\xi\left(\frac{n}{\alpha\left(n\right)}\right)\right)}{\alpha\left(n\right)}\right)\\
= & O\left(\frac{n}{\left(\alpha\left(n\right)\right)^{2}}\log\left(\alpha\left(n\right)+2\right)\right)
\end{align*}
(by Lemma \ref{lem:xiEigensch}) do not contribute to the asymptotic
expansion.

The integrand
\[
\frac{\exp\left(v\right)-1}{v}
\]
in $I$ is strictly increasing in $v$. This fact, $\alpha\left(n\right)\log\left(x_{n,\alpha}\left(1\right)\right)=\xi\left(\frac{n}{\alpha\left(n\right)}\right)+O\left(\frac{\log\left(\frac{n}{\alpha\left(n\right)}+1\right)}{\alpha\left(n\right)}\right)$
(by Lemma \ref{lem:xiEigensch}), and
\[
I^{\prime}\left(\xi\left(u\right)\right)=u
\]
(by Equation (\ref{eq:DefXi})) lead to
\begin{align*}
I\left(\alpha\left(n\right)\log\left(x_{n,\alpha}\left(1\right)\right)\right)= & I\left(\xi\left(\frac{n}{\alpha\left(n\right)}\right)\right)+O\left(\frac{n}{\alpha\left(n\right)}\frac{\log\left(\frac{n}{\alpha\left(n\right)}+1\right)}{\alpha\left(n\right)}\right)\\
= & I\left(\xi\left(\frac{n}{\alpha\left(n\right)}\right)\right)+O\left(\frac{n}{\left(\alpha\left(n\right)\right)^{2}}\log\left(\frac{n}{\alpha\left(n\right)}+1\right)\right).
\end{align*}
where the error term is again of lower order. A calculation similar
to Equation (\ref{eq:ExpIntegral}) then leads to
\[
I\left(\xi\left(\frac{n}{\alpha\left(n\right)}\right)\right)\sim\frac{n}{\alpha\left(n\right)}\sum_{k=0}^{\infty}\frac{k!}{\left(\xi\left(\frac{n}{\alpha\left(n\right)}\right)\right)^{k}}
\]
so that
\[
h_{n,\alpha}^{\prime}\left(1\right)\sim\frac{n}{\alpha\left(n\right)}\sum_{k=0}^{\infty}\frac{k!}{\left(\xi\left(\frac{n}{\alpha\left(n\right)}\right)\right)^{k}}
\]
follows.

As to $h_{n,\alpha}^{\prime\prime}\left(1\right)$, we have
\begin{align}
\alpha\left(n\right)\left(x_{n,\alpha}\left(1\right)-1\right)= & \alpha\left(n\right)\log\left(x_{n,\alpha}\left(1\right)\right)+O\left(\alpha\left(n\right)\left(x_{n,\alpha}\left(1\right)-1\right)^{2}\right)\nonumber \\
= & \xi\left(\frac{n}{\alpha\left(n\right)}\right)+O\left(\frac{\left(\xi\left(\frac{n}{\alpha\left(n\right)}\right)\right)^{2}}{\alpha\left(n\right)}+\frac{\log\left(\frac{n}{\alpha\left(n\right)}+1\right)}{\alpha\left(n\right)}\right)\label{eq:LogXundXminus1}
\end{align}
by Taylor's theorem and Lemma \ref{lem:xiEigensch}. The order of
the error term is such that applying Equation (\ref{eq:LogXundXminus1})
to the result in Proposition \ref{prop:hStrichStrich} yields
\[
h_{n,\alpha}^{\prime\prime}\left(1\right)\sim -\frac{n}{\alpha\left(n\right)}\sum_{k=0}^{\infty}\frac{1}{\left(\xi\left(\frac{n}{\alpha\left(n\right)}\right)\right)^{k}}.
\]
The claim follows.
\end{proof}
\bibliographystyle{plain}

\begin{thebibliography}{10}

\bibitem{AT92}
R.~Arratia and S.~Tavar\'e.
\newblock The cycle structure of random permutations.
\newblock {\em The Annals of Probability}, pages 1567--1591, 1992.

\bibitem{Be14}
V.~Betz.
\newblock Random permutations of a regular lattice.
\newblock {\em Journal of Statistical Physics}, 155(6):1222--1248, 2014.

\bibitem{BU11}
V.~Betz and D.~Ueltschi.
\newblock {Spatial random permutations and Poisson-Dirichlet law of cycle
  lengths}.
\newblock {\em Electron. J. Probab.}, 16(41):41, 2011.

\bibitem{BUV}
V.~Betz, D.~Ueltschi, and Y.~Velenik.
\newblock Random permutations with cycle weights.
\newblock {\em The Annals of Applied Probability}, 21(1):312--331, 2011.

\bibitem{BZ15}
L.~V. Bogachev and D.~Zeindler.
\newblock Asymptotic statistics of cycles in surrogate-spatial permutations.
\newblock {\em Communications in Mathematical Physics}, 334(1):39--116, 2015.

\bibitem{DP85}
J.~Delaurentis and B.~Pittel.
\newblock {Random permutations and Brownian motion}.
\newblock {\em Pacific Journal of Mathematics}, 119(2):287--301, 1985.

\bibitem{NIST:DLMF}
{NIST Digital Library of Mathematical Functions}.
\newblock http://dlmf.nist.gov/6.12.E2, Release 1.0.10 of 2015-08-07.
\newblock Online companion to \cite{Olver:2010:NHMF}.

\bibitem{Flajolet2009}
P.~Flajolet and R.~Sedgewick.
\newblock {\em Analytic Combinatorics}.
\newblock Cambridge University Press, 2009.

\bibitem{goncharov1944field}
V.~L. Goncharov.
\newblock On the field of combinatory analysis.
\newblock {\em Soviet Math. Izv., Ser. Math}, 8:3--48, 1944.

\bibitem{goncharov1944some}
V.~L. Goncharov.
\newblock Some facts from combinatorics.
\newblock {\em Izvestia Akad. Nauk. SSSR, Ser. Mat}, 8:3--48, 1944.

\bibitem{GLU}
S.~Grosskinsky, A.~A. Lovisolo, and D.~Ueltschi.
\newblock {Lattice permutations and Poisson-Dirichlet distribution of cycle
  lengths}.
\newblock {\em Journal of Statistical Physics}, 146(6):1105--1121, 2012.

\bibitem{Han90}
J.~C. Hansen.
\newblock {A functional central limit theorem for the Ewens sampling formula}.
\newblock {\em Journal of Applied Probability}, pages 28--43, 1990.

\bibitem{Kol71}
V.~F. Kolchin.
\newblock A problem of the allocation of particles in cells and cycles of
  random permutations.
\newblock {\em Theory of Probability \& Its Applications}, 16(1):74--90, 1971.

\bibitem{manstavivcius2014local}
E.~Manstavi{\v{c}}ius and R.~Petuchovas.
\newblock Local probabilities for random permutations without long cycles.
\newblock {\em arXiv preprint arXiv:1501.00136}, 2014.

\bibitem{manstavivcius2015permutations}
E.~Manstavi{\v{c}}ius and R.~Petuchovas.
\newblock Permutations without long or short cycles.
\newblock {\em Electronic Notes in Discrete Mathematics}, 49:153--158, 2015.

\bibitem{maples2012number}
K.~Maples, A.~Nikeghbali, and D.~Zeindler.
\newblock On the number of cycles in a random permutation.
\newblock {\em Electron. Commun. Probab.}, 17(20):1--13, 2012.

\bibitem{murray1992asymptotic}
J.~D. Murray.
\newblock {\em Asymptotic Analysis}.
\newblock Applied Mathematical Sciences. Springer New York, 1992.

\bibitem{NiStZei}
A.~Nikeghbali, J.~Storm, and D.~Zeindler.
\newblock Large cycles and a functional central limit theorem for generalized
  weighted random permutations.
\newblock {\em arXiv preprint arXiv:1302.5938}, 2013.

\bibitem{Olver:2010:NHMF}
F.~W.~J. Olver, D.~W. Lozier, R.~F. Boisvert, and C.~W. Clark, editors.
\newblock {\em {NIST Handbook of Mathematical Functions}}.
\newblock Cambridge University Press, New York, NY, 2010.
\newblock Print companion to \cite{NIST:DLMF}.

\bibitem{Sachkov1997}
V.~N. Sachkov.
\newblock {\em Probabilistic Methods in Combinatorial Analysis}.
\newblock Encyclopedia of Mathematics and its Applications. Cambridge
  University Press, 1997.

\bibitem{shepp1966ordered}
L.~A. Shepp and S.~P. Lloyd.
\newblock Ordered cycle lengths in a random permutation.
\newblock {\em Transactions of the American Mathematical Society}, pages
  340--357, 1966.

\bibitem{Yak07}
A.~L. Yakymiv.
\newblock {Random A-permutations: convergence to a Poisson process}.
\newblock {\em Mathematical Notes}, 81(5-6):840--846, 2007.

\bibitem{Yak08}
A.~L. Yakymiv.
\newblock {Limit theorem for the general number of cycles in a random
  A-permutation}.
\newblock {\em Theory of Probability \& Its Applications}, 52(1):133--146,
  2008.

\end{thebibliography}

\end{document}